\RequirePackage{fix-cm}
\documentclass[smallextended]{svjour3}
\smartqed
\usepackage{graphicx}
\usepackage{amsmath,amssymb,amsfonts,latexsym}
\providecommand{\diff}{\mathop{}\!\mathrm{d}}
\newcommand\cL{{\cal L}}
\newcommand{\bbr}{\mathbb{R}}
\newcommand{\bbp}{\mathbb{P}}
\newcommand{\bbe}{\mathbb{E}}
\newcommand{\eps}{\varepsilon}
\usepackage{enumitem}

\begin{document}

\title{Existence of a classical solution to the integro-differential equation arising in the Cram\'er--Lundberg non-life insurance model with proportional investment}

\titlerunning{Existence of a classical solution to the IDE in the Cram\'er--Lundberg}

\author{Platon Promyslov}


\institute{P. Promyslov \at
              \email{platon.promyslov@gmail.com}           
}

\date{Received: date / Accepted: date}

\maketitle

\begin{abstract}
This paper studies the classical Cram\'er--Lundberg non-life insurance model in which an insurance company invests a fixed proportion of its capital in a risky asset whose price follows a geometric Brownian motion. We prove that the survival probability is a classical $C^2$-smooth solution to the corresponding integro-differential equation under minimal moment conditions: it suffices that $\mathbb{E}[\xi^\varepsilon]<\infty$ for some $\varepsilon>0$ and that the claim size distribution function is continuous. The condition $\mathbb{E}[\xi^{\gamma-1}]<\infty$ is required only to derive the power-law asymptotics $\Psi(u)\sim C_\infty u^{-\gamma+1}$, but not to establish the smoothness of the solution. The method relies on reducing the problem to a Volterra integral equation and iteratively refining asymptotic estimates, thereby avoiding unnecessary assumptions.

\end{abstract}

\section{Introduction and Main Results}

This paper studies the ruin problem for an insurance company that invests a fixed proportion of its capital in a risky asset. We assume that the company continuously invests a fixed fraction $\kappa \in (0,1]$ of its capital in a risky asset whose price follows a geometric Brownian motion with drift parameter $a$ and volatility $\sigma > 0$. The remaining fraction $1-\kappa$ is placed in a risk-free bank account with interest rate $r \ge 0$. Then the dynamics of the company's capital $X^u = (X^u_t)_{t \ge 0}$ with initial reserve $u>0$ is governed by the stochastic differential equation:
\begin{equation*}
\diff X^u_t = \big( (a-r)\kappa + r \big) X^u_t \diff t + \kappa \sigma X^u_t \diff W_t + \diff P_t,
\end{equation*}
where $W = (W_t)_{t\ge 0}$ is a standard Wiener process. The process $P_t = ct - \sum_{i=1}^{N_t} \xi_i$ describes the insurance business: $c>0$ is the premium arrival rate, $N = (N_t)_{t\ge 0}$ is a Poisson process with intensity $\lambda>0$, and the claim sizes $(\xi_i)_{i \ge 1}$ are independent positive random variables with distribution function $F$.

The central object of study is the infinite-horizon survival probability $\Phi(u) := \bbp(\tau^u = \infty)$, where $\tau^u := \inf\{t \ge 0 : X^u_t \le 0\}$ denotes the ruin time. Assuming sufficient smoothness of $\Phi(u)$, one can show that it must satisfy the second-order integro-differential equation (IDE):
\begin{multline}
\label{eq:ide}
\frac{1}{2} \kappa^2 \sigma^2 u^2 \Phi''(u) + \big( (a-r)\kappa + r \big) u \Phi'(u) + c \Phi'(u) =
\\
= \lambda \int_0^\infty \big( \Phi(u) - \Phi(u-y) \big) \diff F(y),
\end{multline}
with the natural boundary conditions $\lim_{u\to\infty}\Phi(u) = 1$ and $\Phi(u)=0$ for $u<0$, while the value $\Phi(0+)$ is to be determined.

The main mathematical difficulty in this class of problems is to justify that the survival probability possesses sufficient regularity for \eqref{eq:ide} to be well-defined in the classical sense. In \cite{Grandits2004}, it was shown that reducing the problem to an integral equation requires only a minimal moment condition: $\bbe[\xi^\eps] < \infty$ for some $\eps > 0$. However, that approach guaranteed merely absolute continuity of the derivative $\Phi'(u)$, leaving open the question of $C^2$-smoothness.

In the recent work \cite{AK2026}, existence and uniqueness of a $C^2$-smooth solution to \eqref{eq:ide} were established under the condition $\mathbb{E}[\xi^{\gamma-1}]<\infty$, where the structural parameter
\begin{equation*}
\gamma := \frac{2\big((a-r)\kappa + r\big)}{\kappa^2 \sigma^2} > 1
\end{equation*}
governs the balance between drift and diffusion. The condition $\mathbb{E}[\xi^{\gamma-1}]<\infty$ is evidently violated for heavy-tailed distributions, including those studied in \cite{Grandits2004}. Hence, the framework of \cite{AK2026} does not encompass these important cases.

The aim of the present paper is to show that the requirement of a finite moment of order $\gamma-1$ is superfluous for smoothness itself: $C^2$-regularity can be established under the existence of an arbitrarily small moment $\mathbb{E}[\xi^\varepsilon]<\infty$. The condition $\mathbb{E}[\xi^{\gamma-1}]<\infty$ is needed exclusively for deriving the power-law asymptotics $\Psi(u) := 1 - \Phi(u) \sim C_\infty u^{-\gamma+1}$, but not for the regularity of the solution.

Developing the integral approach proposed for the annuity model in \cite{Promyslov2026}, the present work reduces the IDE \eqref{eq:ide} to a Volterra integral equation on the entire half-line $\bbr_+$. By iteratively improving regularity, we prove absolute integrability of the density of the survival probability. This method allows us to avoid artificial local patching of solutions and the use of majorants depending on the condition $\bbe[\xi^{\gamma-1}] < \infty$.

The analysis relies on the following basic assumptions:
\begin{enumerate}[label=(\textbf{A\arabic*})]
    \item \label{assump:1} The distribution function $F$ is continuous at zero: $F(0)=0$.
    \item \label{assump:2} There exists a constant $\eps > 0$ such that $\bbe[\xi^\eps] < \infty$.
    \item \label{assump:3} The inequality $\gamma > 1$ holds; otherwise $\Psi(u) = 1$ (see \cite{Paulsen1993}[Theorem 3.1]).
\end{enumerate}

The main result of the paper is the following theorem.

\begin{theorem}\label{thm:main}
Suppose that conditions \ref{assump:1}--\ref{assump:3} are satisfied. Then:
\begin{enumerate}[label=(\roman*)]
    \item \textbf{(Existence and smoothness)} The survival probability $\Phi(u)$ belongs to the class $C([0,\infty)) \cap C^1((0,\infty))$ and satisfies the IDE \eqref{eq:ide} almost everywhere. If, in addition, $F$ is continuous on $[0,\infty)$, then $\Phi \in C^2((0,\infty))$ and equation \eqref{eq:ide} holds in the classical sense for all $u > 0$.
    
    \item \textbf{(Boundary conditions)} We have $\lim_{u\to\infty} \Phi(u) = 1$. The right-hand limit at zero is strictly positive: $\Phi(0+) = \Phi_0 \in (0,1)$.

    \item \textbf{(Asymptotics)} If, in addition, $\mathbb{E}[\xi^{\gamma-1}]<\infty$, then there exists a constant $C_\infty>0$ such that  
    \[
    \lim_{u\to\infty} u^{\gamma-1}\Psi(u) = C_\infty.
    \]  
    If, on the contrary, $\mathbb{E}[\xi^{\gamma-1}]=\infty$, then  
    \[
    \lim_{u\to\infty} u^{\gamma-1}\Psi(u) = \infty,
    \]
    and in this case $\Psi(u) \sim \operatorname{const} \bar F(u)$ for subexponential distributions.
\end{enumerate}
\end{theorem}

In addition to the analytical proof of Theorem \ref{thm:main}, Section \ref{sec:identification} provides an explicit integral formula for the constant $C_\infty$ (see \eqref{eq:c_infty}), which allows one to compute the asymptotics without resorting to numerical schemes that may suffer from discretization errors, in particular when approximating differential operators (cf. the approach in \cite{AK2026}).

\section{Reduction to Integral Equations}

We seek a solution to the IDE \eqref{eq:ide} in the class of functions possessing a continuous (or locally absolutely continuous) first derivative $g(u) := \Phi'(u)$. As shown in \cite[Section 3]{Grandits2004}, applying integration by parts to the Lebesgue--Stieltjes integral on the right-hand side of \eqref{eq:ide}, while taking into account that $\Phi(u)=0$ for $u<0$, allows us to rewrite the equation in terms of the function $g$:
\begin{multline}
\label{eq:ide_g_raw}
\frac{1}{2} \kappa^2 \sigma^2 u^2 g'(u) + \big( (a-r)\kappa + r \big) u g(u) + c g(u) =
\\
= \lambda \Phi(0+) \bar{F}(u) + \lambda \int_0^u g(y) \bar{F}(u-y) \diff y,
\end{multline}
where $\bar{F}(u) := 1 - F(u)$ denotes the tail function of the claim size distribution. At this stage, the value $\Phi(0+)$ remains an unknown positive constant.

We introduce the following structural parameters of the model:
\begin{equation*}
\gamma := \frac{2\big((a-r)\kappa + r\big)}{\kappa^2\sigma^2}, \quad \alpha := \frac{2c}{\kappa^2\sigma^2}, \quad \mu := \frac{2\lambda}{\kappa^2\sigma^2}.
\end{equation*}
Dividing both sides of \eqref{eq:ide_g_raw} by $\kappa^2\sigma^2/2$, we obtain a linear first-order integro-differential equation:
\begin{equation}\label{eq:ide_g}
u^2 g'(u) + (\gamma u + \alpha) g(u) = \mu \Phi(0+) \bar{F}(u) + \mu (Bg)(u),
\end{equation}
where $(Bg)(u) := \int_0^u g(u-y) \bar{F}(y) \diff y = \int_0^u g(y) \bar{F}(u-y) \diff y$ is the convolution operator.

Multiplying both sides of \eqref{eq:ide_g} by the integrating factor $u^{\gamma-2} e^{-\alpha/u}$, one readily verifies that the left-hand side collapses to a total derivative, yielding the equation in the form:
\begin{equation}\label{eq:derivative_form}
\frac{\diff}{\diff u} \left( u^\gamma e^{-\alpha/u} g(u) \right) = u^{\gamma-2} e^{-\alpha/u} \mu \big( \Phi(0+) \bar{F}(u) + (Bg)(u) \big).
\end{equation}

Integrating identity \eqref{eq:derivative_form} over the interval $[0, u]$ and noting that, for any function $g$ bounded in a neighborhood of zero, the exponential decay ensures the limiting relation
\[
\lim_{t\downarrow 0} t^\gamma e^{-\alpha/t} g(t) = 0,
\]
we arrive at a Volterra-type integral equation:
\begin{equation}\label{eq:volterra_zero}
g(u) = \frac{\mu}{u^\gamma e^{-\alpha/u}} \int_0^u t^{\gamma-2} e^{-\alpha/t} \big( \Phi(0+) \bar{F}(t) + (Bg)(t) \big) \diff t, \quad u > 0.
\end{equation}
This integral relation serves as the starting point for the analytical proof of global existence of a solution.

\begin{remark}
Equation \eqref{eq:volterra_zero} has a singularity at $u=0$. Therefore, to prove existence of a solution, we first construct a solution on a small interval $[0,u_0]$ (Lemma \ref{lem:local}) via the contraction mapping principle, and then extend it to $[u_0,\infty)$ (\cite[Theorem 2.1.1]{Burton1983}). The two-stage procedure is also employed in \cite{AK2026}; however, for the extension to the half-line, the authors require the additional condition $\mathbb{E}[\xi^{\gamma-1}]<\infty$. Below we show that this condition is superfluous: it suffices that a moment of arbitrarily small order exists, i.e., $\mathbb{E}[\xi^\varepsilon]<\infty$ for some $\varepsilon>0$. The key ingredient is the iterative improvement of the power-law decay rate of $g(u)$, which does not rely on the convergence of the moment of order $\gamma-1$.
\end{remark}

\section{Solvability and A Priori Estimates}
\label{sec:existence}

Due to the presence of the singular factor $u^{-\gamma} e^{\alpha/u}$ at $u=0$, we first establish the existence of a solution on a sufficiently small interval.

\begin{lemma}[Local Solvability]\label{lem:local}
Under condition \ref{assump:1}, for any $\Phi(0+) > 0$ there exists a sufficiently small $u_0 > 0$ such that the integral equation \eqref{eq:volterra_zero} has a unique continuous solution $g \in C([0, u_0])$ satisfying $g(0) = \lambda \Phi(0+)/c$.
\end{lemma}
\begin{proof}
We rewrite \eqref{eq:volterra_zero} in the form $g = \mathcal{T}g$, where
\begin{align*}
    (\mathcal{T}g)(u) &= \frac{\mu}{u^\gamma e^{-\alpha/u}} \int_0^u t^{\gamma-2} e^{-\alpha/t} \bigl( \Phi(0+)\bar F(t) + (Bg)(t) \bigr) \diff t,
    \\
    (Bg)(t) &= \int_0^t g(z)\bar F(t-z) \diff z.
\end{align*}
Note that for any $g_1, g_2 \in C([0,u])$ we have $|(Bg_1)(t) - (Bg_2)(t)| \le t \|g_1-g_2\|_{C[0,u]}$. We estimate the difference of the operators:
\begin{multline*}
    |(\mathcal{T}g_1)(u) - (\mathcal{T}g_2)(u)|
    \le \frac{\mu}{u^\gamma e^{-\alpha/u}} \int_0^u t^{\gamma-2} e^{-\alpha/t} \cdot t \|g_1-g_2\|_{C[0,u]} \diff t =
    \\
    = \frac{\mu \|g_1-g_2\|_{C[0,u]}}{u^\gamma e^{-\alpha/u}} \int_0^u t^{\gamma-1} e^{-\alpha/t} \diff t.
\end{multline*}
The change of variables $s = \alpha/t$ yields the precise asymptotics of the integral as $u \to 0$:
\[
\int_0^u t^{\gamma-1} e^{-\alpha/t} \diff t = \alpha^{\gamma} \int_{\alpha/u}^{\infty} s^{-\gamma-1} e^{-s} \diff s \sim \frac{u^{\gamma+1}}{\alpha} e^{-\alpha/u}.
\]
Consequently, there exists a constant $C_0 > 0$ such that for all sufficiently small $u$ we have
\[
\frac{1}{u^\gamma e^{-\alpha/u}} \int_0^u t^{\gamma-1} e^{-\alpha/t} \diff t \le C_0 u.
\]
Thus, we arrive at the estimate:
\[
\|\mathcal{T}g_1 - \mathcal{T}g_2\|_{C[0,u]} \le \mu C_0 u \|g_1-g_2\|_{C[0,u]}.
\]
By choosing $u_0$ such that $\mu C_0 u_0 < 1$, we ensure that the operator $\mathcal{T}$ is a contraction. By the Banach fixed point theorem, there exists a unique solution $g \in C([0,u_0])$. Passing to the limit $u\to 0$ in \eqref{eq:volterra_zero} and applying L'Hôpital's rule yields $g(0) = \frac{\lambda}{c}\Phi(0+)$.
\end{proof}

Fix the found $u_0>0$ and the values of $g$ on the interval $[0, u_0]$. For $u \ge u_0$, we split the integral in \eqref{eq:volterra_zero} into two parts: from $0$ to $u_0$ and from $u_0$ to $u$. We separate the contribution of $g$ on $[0,u_0]$ and on $[u_0,t]$ in $(Bg)(t)$:
\[
(Bg)(t) = \int_0^{u_0} g(z)\bar F(t-z) \diff z + \int_{u_0}^t g(z)\bar F(t-z) \diff z.
\]
Substituting this into \eqref{eq:volterra_zero}, we obtain for $u \ge u_0$ a linear Volterra integral equation of the second kind:
\begin{equation}\label{eq:volterra_global}
g(u) = f_{u_0}(u) + \int_{u_0}^u K(u,y) g(y) \diff y,
\end{equation}
where the free term $f_{u_0}(u)$ depends only on the already known values of $g$ on $[0, u_0]$:
\begin{multline*}
f_{u_0}(u) = \frac{\mu}{u^\gamma e^{-\alpha/u}} \Biggl( \int_0^{u_0} t^{\gamma-2} e^{-\alpha/t} \Bigl( \Phi(0+)\bar F(t) + \int_0^t g(z)\bar F(t-z) \diff z \Bigr) \diff t +
\\
+ \int_{u_0}^u t^{\gamma-2} e^{-\alpha/t} \int_0^{u_0} g(z)\bar F(t-z) \diff z \diff t \Biggr),
\end{multline*}
and the kernel is given by:
\[
K(u,y) = \frac{\mu}{u^\gamma e^{-\alpha/u}} \int_y^u t^{\gamma-2} e^{-\alpha/t} \bar F(t-y) \diff t.
\]
The function $f_{u_0}(u)$ is continuous on $[u_0,\infty)$, and the kernel $K(u,y)$ is jointly continuous on the set $u_0 \le y \le u$. By \cite[Theorem 2.1.1]{Burton1983}, equation \eqref{eq:volterra_global} has a unique continuous solution $g \in C([u_0,\infty))$. Thus, the solution $g$ is well-defined on the entire half-line $[0,\infty)$.

\begin{lemma}[Uniform Boundedness]\label{lem:boundedness}
Suppose condition \ref{assump:2} holds. Then the continuous solution $g(u)$ is uniformly bounded on $[0,\infty)$.
\end{lemma}
\begin{proof}
Define the non-decreasing function $M(u) := \max_{x \in [0,u]} |g(x)|$. From the definition of the convolution operator, we have:
\[
|(Bg)(t)| \le \int_0^t |g(t-y)| \bar F(y) \diff y \le M(u) \int_0^t \bar F(z) \diff z, \quad \text{for all } t \le u.
\]
Denote $J(t) := \int_0^t \bar F(z) \diff z$. Taking absolute values in \eqref{eq:volterra_zero}, we obtain:
\begin{multline*}
|g(u)| \le \frac{\mu}{u^\gamma e^{-\alpha/u}} \int_0^u t^{\gamma-2} e^{-\alpha/t} \Phi(0+)\bar F(t) \diff t +
\\
+ M(u) \frac{\mu}{u^\gamma e^{-\alpha/u}} \int_0^u t^{\gamma-2} e^{-\alpha/t} J(t) \diff t =: \tilde f(u) + M(u) I_2(u),
\end{multline*}
where $\tilde f(u)$ is a bounded function (its limit as $u\to\infty$ is zero, and as $u\to 0$ it equals $\lambda\Phi(0+)/c$). Hence, there exists a constant $K_f$ such that $\tilde f(u) \le K_f$ for all $u\ge0$.

By condition \ref{assump:2}, there exists a moment $\mathbb{E}[\xi^\varepsilon]<\infty$ for some $\varepsilon>0$. If necessary, we reduce $\varepsilon$ so that $\varepsilon\in(0,1)$ and $\gamma-1-\varepsilon>0$. By Markov's inequality, $\bar F(z) \le C_0 z^{-\varepsilon}$ for all $z>0$, where $C_0 = \mathbb{E}[\xi^\varepsilon]$. We estimate $J(t)$:
\begin{multline*}
J(t) = \int_0^1 \bar F(z) \diff z + \int_1^t \bar F(z) \diff z \le 1 + C_0 \int_1^t z^{-\varepsilon} \diff z =
\\
= 1 + \frac{C_0}{1-\varepsilon}(t^{1-\varepsilon}-1) \le C_1 t^{1-\varepsilon}
\end{multline*}
for all $t \ge 1$, where $C_1 = \max\bigl(1, \frac{C_0}{1-\varepsilon}\bigr)$. For $t \in[0,1)$ we have $J(t) \le \int_0^t 1 \diff z = t \le t^{1-\varepsilon}$ (since $1-\varepsilon \in (0,1)$). Thus, the estimate $J(t) \le C_1 t^{1-\varepsilon}$ holds for all $t \ge 0$. Substituting it into $I_2(u)$, we obtain:
\[
I_2(u) \le \frac{\mu C_1}{u^\gamma e^{-\alpha/u}} \int_0^u t^{\gamma-1-\varepsilon} e^{-\alpha/t} \diff t.
\]
We find the asymptotics as $u\to\infty$ using L'Hôpital's rule:
\begin{multline*}
\lim_{u\to\infty} \frac{\int_0^u t^{\gamma-1-\varepsilon} e^{-\alpha/t} \diff t}{u^\gamma e^{-\alpha/u}}
= \lim_{u\to\infty} \frac{u^{\gamma-1-\varepsilon} e^{-\alpha/u}}{(\gamma u^{\gamma-1} + \alpha u^{\gamma-2}) e^{-\alpha/u}}
=
\\
= \lim_{u\to\infty} \frac{u^{-\varepsilon}}{\gamma + \alpha u^{-1}} = 0.
\end{multline*}
Consequently, there exists $u_1 \ge u_0$ such that $I_2(u) \le 1/2$ for all $u \ge u_1$. Then for $u \ge u_1$:
\[
|g(u)| \le K_f + \frac{1}{2} M(u).
\]
Let $x^* \in [0,u]$ be a point where $|g|$ attains its maximum $M(u)$. If $x^* \le u_1$, then $M(u) = M(u_1)$. If $x^* > u_1$, applying the inequality to $x^*$ yields
\[
M(u) = |g(x^*)| \le K_f + \frac{1}{2} M(x^*) \le K_f + \frac{1}{2} M(u),
\]
whence $M(u) \le 2K_f$. Thus, for all $u \ge 0$ we have $M(u) \le \max(M(u_1), 2K_f)$, which means that $g$ is uniformly bounded on $[0,\infty)$.
\end{proof}

We now prove that $g$ is not only bounded but also absolutely integrable on $\mathbb{R}_+$. The key ingredient is the iterative improvement of the asymptotic estimate, based on the power-law decay of the tail $\bar F$.

\begin{lemma}[Absolute Integrability]\label{lem:iis}
Suppose conditions \ref{assump:2} and \ref{assump:3} hold, with $\varepsilon \in (0,1)$ chosen such that $\gamma-1-\varepsilon > 0$. Then $g(u) = O(u^{-1-\varepsilon})$ as $u \to \infty$. In particular, $g \in L^1(\mathbb{R}_+)$.
\end{lemma}
\begin{proof}
From Lemma \ref{lem:boundedness}, we know that $|g(u)| \le C_0$ for all $u\ge0$. Assume that for some $\delta \in [0,1]$ we have
\[
|g(u)| \le C (1+u)^{-\delta} \qquad \text{for all } u \ge 0.
\]
For $\delta = 0$ this holds with $C = C_0$. We show that there then exists a constant $C'$, independent of $u$, such that $|g(u)| \le C' (1+u)^{-(\delta+\varepsilon)}$ for all sufficiently large $u$.

We estimate the convolution $(Bg)(t)$ for large $t$. We split the integral:
\[
|(Bg)(t)| \le \int_0^{t/2} |g(t-y)| \bar F(y) \diff y + \int_{t/2}^t |g(y)| \bar F(t-y) \diff y.
\]
For the first integral, $y \le t/2$, so $t-y \ge t/2$ and $|g(t-y)| \le C (1+t/2)^{-\delta} \le C_1 t^{-\delta}$. Given that $\int_0^{t/2} \bar F(y) \diff y \le J(t) \le C_2 t^{1-\varepsilon}$ (see the proof of Lemma \ref{lem:boundedness}), the first integral does not exceed $C_3 t^{1-\delta-\varepsilon}$. For the second integral: for $y \ge t/2$ we have $|g(y)| \le C_1 t^{-\delta}$, and $\int_{t/2}^t \bar F(t-y) \diff y = \int_0^{t/2} \bar F(z) \diff z \le C_2 t^{1-\varepsilon}$. Consequently, the second integral is also bounded by $C_3 t^{1-\delta-\varepsilon}$. Thus, there exists a constant $K_1$ such that for all sufficiently large $t$ we have
\[
|(Bg)(t)| \le K_1 t^{1-\delta-\varepsilon}.
\]

Moreover, for $t \ge 1$ and since $1-\delta \ge 0$, we have $\bar F(t) \le C_0 t^{-\varepsilon} \le C_0 t^{1-\delta-\varepsilon}$. Substituting these estimates into \eqref{eq:volterra_zero}, we obtain for large $u$:
\begin{multline*}
|g(u)| \le \frac{\mu}{u^\gamma e^{-\alpha/u}} \int_0^u t^{\gamma-2} e^{-\alpha/t} \bigl( \Phi(0+)\bar F(t) + |(Bg)(t)| \bigr) \diff t \le
\\
\le \frac{\mu K_2}{u^\gamma e^{-\alpha/u}} \int_0^u t^{\gamma-2} e^{-\alpha/t} t^{1-\delta-\varepsilon} \diff t
= \frac{\mu K_2}{u^\gamma e^{-\alpha/u}} \int_0^u t^{\gamma-1-\delta-\varepsilon} e^{-\alpha/t} \diff t.
\end{multline*}
The asymptotics of the last integral as $u\to\infty$ is given by:
\[
\int_0^u t^{\gamma-1-\delta-\varepsilon} e^{-\alpha/t} \diff t \sim \frac{u^{\gamma-\delta-\varepsilon}}{\gamma-\delta-\varepsilon} e^{-\alpha/u},
\]
which is easily derived using L'Hôpital's rule. Taking into account the cancellation of factors, we get:
\[
|g(u)| \le \frac{\mu K_2}{\gamma-\delta-\varepsilon} u^{-(\delta+\varepsilon)} + o(u^{-(\delta+\varepsilon)}).
\]
Hence, for all sufficiently large $u$ we have $|g(u)| \le C' u^{-(\delta+\varepsilon)}$, as required.

Starting from $\delta = 0$, after $N = \lceil 1/\varepsilon \rceil$ steps we reach $\delta \ge 1$. We apply one more step, starting from $\delta = 1$. In this case $|g(u)| \le C u^{-1}$ and, as shown above, $|(Bg)(t)| \le K_3 t^{-\varepsilon}$. Then
\begin{multline*}
|g(u)| \le \frac{\mu}{u^\gamma e^{-\alpha/u}} \int_0^u t^{\gamma-2} e^{-\alpha/t} \bigl( \Phi(0+)\bar F(t) + K_3 t^{-\varepsilon} \bigr) \diff t
\le
\\
\le \frac{\mu K_4}{u^\gamma e^{-\alpha/u}} \int_0^u t^{\gamma-2-\varepsilon} e^{-\alpha/t} \diff t.
\end{multline*}
The asymptotics yields:
\[
|g(u)| \le \frac{\mu K_4}{\gamma-1-\varepsilon} u^{-1-\varepsilon} + o(u^{-1-\varepsilon})
\]
for all sufficiently large $u$. Thus, $g(u) = O(u^{-1-\varepsilon})$, which, since $\varepsilon>0$, implies $g \in L^1(\mathbb{R}_+)$.
\end{proof}

Finally, we note that from the integral representation \eqref{eq:volterra_zero}, the strict positivity of $\Phi(0+)$, and the non-negativity of $\bar F(t)$, it follows that $g(u) > 0$ for all $u>0$. Indeed, on the small interval $[0,u_0]$ the solution is obtained as the limit of Picard iterations, which preserve strict positivity; for $u \ge u_0$, equation \eqref{eq:volterra_global} with a strictly positive free term and a non-negative kernel yields a positive solution. Consequently, $\Phi'(u) = g(u) > 0$, so that $\Phi$ is strictly increasing, which is fully consistent with the probabilistic meaning of the problem.

\section{Solution Verification and Asymptotics of the Ruin Probability}
\label{sec:identification}

Let $g_1(u)$ denote the solution to the integral equation \eqref{eq:volterra_zero} corresponding to the value $\Phi(0+) = 1$. Due to the linearity of the convolution operator $B$ and the homogeneity of the right-hand side with respect to $\Phi(0+)$, the general solution takes the form $g(u) = \Phi(0+) g_1(u)$. 

We construct a candidate function for the survival probability:
\begin{equation}\label{eq:G_candidate}
G(u) := \Phi(0+) + \int_0^u g(y) \diff y = \Phi(0+) \left( 1 + \int_0^u g_1(y) \diff y \right).
\end{equation}
By Lemma \ref{lem:iis}, the integral $I_1 := \int_0^\infty g_1(y) \diff y$ is finite. The normalization condition at infinity, $G(\infty)=1$, uniquely determines the initial value:
\begin{equation}\label{eq:phi_zero}
\Phi(0+) = \frac{1}{1 + I_1} \in (0,1).
\end{equation}
The strict positivity of $\Phi(0+)$ reflects the fact that in the classical non-life model with premium rate $c>0$, the company cannot be ruined instantaneously even with zero initial capital.

\begin{proposition}
The function \eqref{eq:G_candidate} coincides with the survival probability $\Phi$.
\end{proposition}
\begin{proof}
We extend the function to the entire real line by setting $\widetilde G(x) := G(x)$ for $x > 0$ and $\widetilde G(x) := 0$ for $x \le 0$. By construction, $\widetilde G$ is bounded, its restriction to $(0,\infty)$ belongs to $C^1((0,\infty))$, and its derivative is locally absolutely continuous on the positive half-line. We apply the generalized Itô formula to the process $\widetilde G(X_{t\wedge\tau^u}^u)$:
\begin{equation*}
\widetilde G(X_{t\wedge\tau^u}^u) = \widetilde G(u) + \int_0^{t\wedge\tau^u} (\cL \widetilde G)(X_s^u) \diff s + M_{t\wedge\tau^u},
\end{equation*}
where $M_t$ is a local martingale and $\cL \widetilde G$ is the integro-differential operator. Up to the stopping time $\tau^u$, the process remains in the domain of strictly positive values, where the operator coincides with the left-hand side of equation \eqref{eq:ide_g_raw} minus the right-hand side. Since $G$ is constructed from a solution to \eqref{eq:ide_g_raw}, the time integral vanishes identically. The process $M_{t\wedge\tau^u}$ is bounded (since $0 \le \widetilde G \le 1$), hence it is a martingale and uniformly integrable. Therefore, $\mathbb{E}[\widetilde G(X_{t\wedge\tau^u}^u)] = \widetilde G(u) = G(u)$.

Letting $t\to\infty$, on the set $\{\tau^u<\infty\}$ ruin occurs, meaning the process reaches the non-positive half-line: $X_{\tau^u}^u \le 0$. By definition, $\widetilde G(X_{\tau^u}^u)=0$. On the set $\{\tau^u=\infty\}$, the process survives. As is well known (see \cite[Theorem 3.1]{Paulsen1993}), the condition $\gamma>1$ ensures that the drift dominates the diffusion, which implies $X_t^u \to \infty$ almost surely as $t\to\infty$. Consequently, $\widetilde G(X_t^u) \to G(\infty)=1$. By the Lebesgue dominated convergence theorem, we obtain:
\begin{equation*}
G(u) = 0 \cdot \mathbb{P}(\tau^u<\infty) + 1 \cdot \mathbb{P}(\tau^u=\infty) = \Phi(u).
\end{equation*}
\end{proof}

We now turn to the analysis of the behavior of $\Psi(u) = 1 - \Phi(u)$ as $u \to \infty$. By L'Hôpital's rule,
\begin{equation*}
\lim_{u\to\infty} u^{\gamma-1} \Psi(u) = \frac{1}{\gamma-1} \lim_{u\to\infty} u^\gamma g(u).
\end{equation*}
From the integral representation \eqref{eq:volterra_zero}, we have:
\begin{equation}\label{eq:limit_integral}
u^\gamma e^{-\alpha/u} g(u) = \mu \int_0^u t^{\gamma-2} e^{-\alpha/t} \big( \Phi(0+) \bar{F}(t) + (Bg)(t) \big) \diff t.
\end{equation}
Since the integrand is non-negative, the RHS of \eqref{eq:limit_integral} is a non-decreasing function of the upper limit. Because the factor $e^{-\alpha/u} \to 1$ as $u \to \infty$, the limit $L := \lim_{u\to\infty} u^\gamma g(u)$ is guaranteed to exist (finite or infinite) and equals the value of the corresponding improper integral over the half-line.

It is known (see, e.g., \cite{Frolova2002,Paulsen1993}) that the condition $\mathbb{E}[\xi^{\gamma-1}] < \infty$ is necessary and sufficient for the finiteness of the limit $L$. Indeed, the integrability of the term $t^{\gamma-2} \bar{F}(t)$ is equivalent to the convergence of the moment of order $\gamma-1$, and the finiteness of the contribution from the convolution operator $(Bg)(t)$ under this condition follows from standard estimates. 

Consequently, under the condition $\mathbb{E}[\xi^{\gamma-1}] < \infty$, the ruin probability exhibits power-law asymptotics $\Psi(u) \sim C_\infty u^{-\gamma+1}$, where the constant $C_\infty = L/(\gamma-1)$ can be written explicitly using \eqref{eq:phi_zero}:
\begin{equation}\label{eq:c_infty}
C_\infty = \frac{1}{\gamma-1}\, \frac{\mu}{1+\int_0^\infty g_1(y)\diff y} \int_0^\infty t^{\gamma-2} e^{-\alpha/t} \big( \bar{F}(t) + (Bg_1)(t) \big) \diff t.
\end{equation}
Note that the exponential factor $e^{-\alpha/t}$ inside the integral plays a crucial role, ensuring convergence at zero for all admissible values of $\gamma > 1$.

If, however, $\mathbb{E}[\xi^{\gamma-1}] = \infty$, then the integral in \eqref{eq:limit_integral} diverges, which implies $\lim_{u\to\infty} u^{\gamma-1}\Psi(u) = \infty$. In this case, the asymptotics of the ruin probability is no longer purely power-law and is entirely determined by the heavy-tail properties of $\bar{F}(u)$ (see \cite{Grandits2004}).

\section{Smoothness of the Solution}

The $C^2$-smoothness of the survival probability $\Phi$ is established below, completing the proof of Theorem~\ref{thm:main}. Traditionally, this step has required rather restrictive assumptions in the literature: for instance, \cite{KP2022} required the densities of positive and negative jumps to be twice continuously differentiable on $(0,\infty)$, with their first and second derivatives belonging to $L^1(\mathbb{R}_+)$. In the recent work \cite{AK2026}, smoothness was obtained under the stringent moment condition $\mathbb{E}[\xi^{\gamma-1}] < \infty$.

Note that the $C^2$-smoothness can also be justified using the theory of viscosity solutions (see, e.g., \cite{BK2015}). For the annuity model this was done in the recent work \cite{Promyslov2026_Viscosity} by means of local elliptic regularity. In contrast to that approach, in the present paper for the non-life insurance model the $C^2$-smoothness follows directly from the structure of the first-order integral equation, without invoking the viscosity apparatus and without additional moment conditions.

From Section \ref{sec:existence}, we know that $g(u) = \Phi'(u)$ is a continuous and bounded function on $(0,\infty)$ satisfying the Volterra integral equation \eqref{eq:volterra_zero}. This equation can be rewritten as:
\begin{equation}\label{eq:H_integral}
u^\gamma e^{-\alpha/u} g(u) = \mu \int_0^u t^{\gamma-2} e^{-\alpha/t} \big( \Phi(0+) \bar{F}(t) + (Bg)(t) \big) \diff t.
\end{equation}
Denote the left-hand side by $H(u) := u^\gamma e^{-\alpha/u} g(u)$. Since $g$ is continuous and $\bar{F}$ is bounded and locally integrable, the convolution $(Bg)(t) = \int_0^t g(t-y)\bar{F}(y)\diff y$ defines a continuous function of $t$. The tail function $\bar{F}(t)$, being monotone, is continuous almost everywhere and locally integrable. 

Consequently, the integrand in \eqref{eq:H_integral} is locally integrable on $(0,\infty)$ and bounded on any compact set separated from zero. By the properties of the Lebesgue integral with a variable upper limit, it immediately follows that the function $H(u)$ is locally absolutely continuous. Since the factor $u^{-\gamma} e^{\alpha/u}$ is infinitely differentiable on $(0,\infty)$, the function itself $g(u) = u^{-\gamma} e^{\alpha/u} H(u)$ is also locally absolutely continuous ($g \in W^{1,1}_{\text{loc}}((0,\infty))$). Hence, its weak derivative $g'(u)$ exists almost everywhere.

Differentiating identity \eqref{eq:H_integral} almost everywhere necessarily returns us to the original IDE \eqref{eq:ide_g}. Solving for $g'(u)$ by dividing the equation by $u^2$ yields:
\begin{equation}\label{eq:g_prime}
g'(u) = -\frac{\gamma u + \alpha}{u^2}\, g(u) + \frac{\mu}{u^2}\int_0^u g(u-y)\bar{F}(y)\diff y + \frac{\mu \Phi(0+)}{u^2}\,\bar{F}(u) \quad \text{a.e. on } u>0.
\end{equation}

Let us analyze the terms on the right-hand side of \eqref{eq:g_prime}. The first term is continuous because $g \in C((0,\infty))$. The second term (the convolution of the continuous function $g$ and the bounded integrable function $\bar{F}$) also defines a continuous function. 

Consequently, the only possible source of discontinuities in the weak derivative $g'(u)$ is the tail function $\bar{F}(u)$ appearing in the third term. This leads to the following exhaustive classification:
\begin{itemize}
    \item If the claim size distribution function $F$ is continuous on $(0,\infty)$, then $\bar{F}$ is also continuous. In this case, the right-hand side of \eqref{eq:g_prime} is entirely continuous. Since a locally absolutely continuous function whose weak derivative admits a continuous representative is continuously differentiable, we have $g' \in C((0,\infty))$. This implies that $\Phi \in C^2((0,\infty))$, and the original IDE \eqref{eq:ide} holds in the strict classical sense for all $u>0$.
    \item If the distribution $F$ contains atoms, then $\bar{F}$ has a countable number of jump discontinuities. As rigorously shown above, $g \in W^{1,1}_{\text{loc}}((0,\infty))$, meaning that the survival probability $\Phi$ possesses a locally absolutely continuous first derivative, and the IDE \eqref{eq:ide} is satisfied almost everywhere with respect to the Lebesgue measure.
\end{itemize}

This completes the proof of part (i) of Theorem \ref{thm:main}.

In conclusion, we note that establishing $C^2$-smoothness required only the condition $\mathbb{E}[\xi^\varepsilon] < \infty$, which guarantees the existence and boundedness of the continuous function $g$, together with the continuity of $F$. The convergence of higher-order moments, in particular $\mathbb{E}[\xi^{\gamma-1}] < \infty$, does not enter at any stage of the smoothness analysis. This rigorously confirms the superfluity of the condition imposed in \cite{AK2026} to justify a classical solution.

\section*{Competing interests}
The author declares no competing interests.


\end{document}